\documentclass[11pt]{amsart}

\usepackage{amsmath}

\usepackage{amssymb}

\newtheorem{theorem}{Theorem}[section]

\newtheorem{corollary}[theorem]{Corollary}

\newtheorem{lemma}[theorem]{Lemma}

\newtheorem{proposition}[theorem]{Proposition}

\newtheorem*{thmA*}{Theorem A}
\newtheorem*{thmB*}{Theorem B}
\newtheorem*{ack*}{Acknowledgements}

\theoremstyle{definition}
\newtheorem{definition}[theorem]{Definition}
\newtheorem{example}[theorem]{Example}
\newtheorem{remark}[theorem]{Remark}

\DeclareMathOperator{\End}{End} 
\DeclareMathOperator{\diag}{diag}

\DeclareMathOperator{\rann}{r{.}ann}
\DeclareMathOperator{\lann}{l{.}ann}

\DeclareMathOperator{\Hom}{Hom}

\newcommand\NN{\mathbb{N}}
\newcommand\RR{\mathbb{R}}
\newcommand\CC{\mathbb{C}}
\newcommand\CF{CF}
\newcommand\RF{RF}
\newcommand\BR{B(\RR)}
\newcommand\BC{B(\CC)}
\newcommand\minff{M}
\newcommand\minfr{M}
\newcommand\ZZ{\mathbb{Z}}
\newcommand\Zpos{\ZZ_{>0}}
\newcommand\RMC{R$\cdot$M$\cdot$C}
\DeclareMathOperator{\lc}{l{.}coeff}
\newcommand\smallsharp{\text{\raisebox{1pt}{$\scriptstyle\#$}}}
\newcommand\smallsharpp{\text{\raisebox{1pt}{$\scriptstyle\#'$}}}

\numberwithin{theorem}{section}
\numberwithin{equation}{theorem}

\begin{document}

\title{Positive definite matrices and involutions:\\ the manners of their infinite cousins}
\vspace{5mm}

\author[Ara]{Pere Ara}
\address{Departament de Matem\`atiques, Edifici Cc, Universitat Aut\`onoma de Barcelona, 08193 Cerdanyola del Vall\`es (Barcelona), Spain}
\email{pere.ara@uab.cat}

\author[Goodearl]{\ Ken  Goodearl}
\address{Department of Mathematics, University of California, Santa Barbara, CA 93106, USA}
\email{goodearl@math.ucsb.edu}

\author[O'Meara]{\ Kevin C. O'Meara}
\address{University of Canterbury, Christchurch, New Zealand}
\email{kco81@proton.me}

\begin{abstract}
Matrices and involutions in/on the algebras $\BR$ and $\BC$ of real (resp., complex) row- and column-finite $\omega\times\omega$ matrices are studied.  It is proved that any positive definite $\RR$-algebra involution on $\BR$ (resp., any positive definite conjugate-linear involution on $\BC$) is given by conjugating the transpose involution (resp., the conjugate-transpose involution) with a positive definite matrix from the algebra in question.  All positive definite matrices in these algebras have Cholesky factorizations within the algebra.  Examples are constructed to show that such positive definite matrices need not have any eigenvalues for their canonical action on column-finite column vectors, and they need not have any square roots in $\BC$. 
\end{abstract}



\thanks{The first-named author was partially supported by the Spanish State Research Agency (grant No.\ PID2023-147110NB-I00)}

\maketitle

\section{Introduction}

Deeply rooted though the concepts such as positive definite finite matrices are, tracing back to Henry J. Smith, James Sylvester, Arthur Cayley, and other mathematicians of the 19th century, extending these concepts to infinite matrices and working with them face a number of hurdles, as is well known. Here we are not even guaranteed associativity of multiplication (when defined), or any eigenvalues of a real symmetric matrix, for instance.

The algebra $B(F)$ of row- and column-finite matrices over a field $F$ provides a nice halfway house between the algebras $M_n(F)$ of finite matrices and the very large algebras of infinite matrices. Ring theoretically, three particular properties make $B(F)$ attractive: (i) it contains (isomorphic) copies of all countable-di\-men\-sion\-al algebras over $F$ \cite[Proposition 2.1]{GMM}; (ii) although not a (von Neumann) regular ring (so not every principal right ideal is generated by an idempotent), $B(F)$ is an exchange ring \cite[Theorem 1]{OM2003} (whence for each $a \in B(F)$, there is an idempotent $e \in B(F)$ with $e B(F) \subseteq aB(F)$ and $(1-e)B(F) \subseteq (1-a)B(F)$); (iii) $B(F)$ has a natural involution (the transpose). 
\medskip

One of our main objectives is a complete description of all involutions (anti-automorphisms of order at most $2$) on the algebras $\BR$ and $\BC$.  We show that positive definite involutions on $\BR$ (resp., on $\BC$) are essentially the same as the transpose (resp., the conjugate-transpose) involution $*$, in that such involutions arise from $*$ via conjugation by positive definite matrices.  More precisely, we have two main groups of results, which we combine in Theorems A and B.

\begin{thmA*}
{\rm (See Theorems \ref{T:EquivInvol}, \ref{T:Equiv*}, Corollaries \ref{proper.inv.BR}, \ref{proper.inv.BC})}
Let $F=\RR$ or $\CC$, and let $\smallsharp$ be an $\RR$-linear involution on $B(F)$ if $F=\RR$ or a conjugate-linear involution on $B(F)$ if $F=\CC$.  Then the following are equivalent:

{\rm(a)} $\smallsharp$ is proper, i.e., $A^{\#}A = 0$ implies $A=0$, for $A \in B(F)$.

{\rm(b)} $\smallsharp$ is positive definite, i.e., $\sum_{i=1}^n A_i^{\#}A_i = 0$ implies all $A_i = 0$, for $A_i \in B(F)$.

{\rm(c)} There is an invertible positive definite matrix $C \in B(F)$ such that $A^{\#} = C^{-1}A^*C$ for all $A \in B(F)$.
\qed\end{thmA*}

Theorem A also holds in the finite case, i.e., with $B(F)$ replaced by $M_n(F)$.
Development of this theorem requires work with positive definite matrices in $B(F)$, so another of our aims is the study of those.  In particular:

\begin{thmB*}
{\rm(See Theorems \ref{posdef.vs.Chofac}, \ref{Ch.fac.inBR})}
Let $F=\RR$ or $\CC$, and let $A \in B(F)$ be positive definite, i.e., $x^*Ax > 0$ for all nonzero column vectors $x$ with finitely many nonzero entries.

{\rm(a)} $A$ has a unique Cholesky factorization within $B(F)$, i.e., $A = U^*U$ for a unique upper triangular matrix $U \in B(F)$ with all diagonal entries positive.

{\rm(b)} $A$ is invertible in $B(F)$ if and only if $U$ is invertible in $B(F)$.
\qed\end{thmB*}

Unlike the situation for finite matrices, positive definite matrices in $\BR$ or $\BC$ need not be invertible, need not have any eigenvalues, and need not have any square roots -- see Examples \ref{unit.posdef.no.eigenval}--\ref{tridiag.posdef.no.eigenval}, \ref{tridiag.posdef.no.sqrt}, \ref{unit.posdef.no.sqrt}.

Underlying many ``bad manners'' of infinite matrices is the fact that matrix products, even when all terms are defined, are not necessarily associative.  We pinpoint a number of aspects of this problem in Section \ref{sec:assoc} and \S\ref{subsec:invert}.

\subsection{}
There are many instances in mathematics where countably-infinite matrices (over a field $F$) are naturally indexed as $I\times I$ matrices for some countable set $I$ other than $\NN$. The ring of all $I\times I$ column-finite matrices is of course then ring isomorphic to our $\CF(F)$ below. However, in our particular setting, given the connection of involutions to positive definite matrices, and hence in turn to the ordering of the real numbers, we find it safer to use a well-ordered index set. (We have chosen the smallest such on $\NN$.) Besides, one relies on the well-ordering of rows and columns whenever multiplying two displayed infinite matrices! Readers are of course free to adopt more general index sets than in our setting but in some cases could risk voiding their warranties if due care is not exercised (e.g. claiming uniqueness of Cholesky factorizations in Theorem \ref{posdef.vs.Chofac} for general index sets; see \cite[Proposition 2.2]{CWS}).

\subsection{Notation and Conventions}  \label{NotaConv}
We consider several types of (countably) infinite matrices, with row and/or column index sets which are finite or of order type $\omega$.  These sets will be indexed by integer intervals $[1,n]$ (in the finite case) or by the positive integers (in the infinite case).  Identity matrices will be denoted by $I$, or by $I_n$ or $I_\omega$ if the size needs to be recorded.  The transpose of a matrix $A$ or vector $v$ will be denoted $A^t$ or $v^t$, respectively.

Products of matrices (when defined) are always by matrix multiplication, as are actions of such matrices on row or column vectors.  We do not assume any topology (other than the discrete topology) on our base field, so \emph{matrix products are based on discrete row-by-column products}.  Namely, the product of an infinite row $r = [r_1,r_2,\dots]$ times an infinite column $c = [c_1,c_2,\dots]^t$ is defined if and only if $r_jc_j = 0$ for all but finitely many $j$, in which case $rc = \sum_{j=1}^\infty r_jc_j$ amounts to a finite sum.

Fix a field $F$, and let $\minff = \minff(F)$ denote the vector space of all $\omega\times\omega$ matrices over $F$.  There is no matrix-type multiplication operation on all of $\minff$, but various subspaces of $\minff$ support matrix multiplication and are associative algebras over $F$, namely: The space $\CF = \CF(F)$ of column-finite matrices in $\minff$, the space $\RF = \RF(F)$ of row-finite matrices in $\minff$, and the space $B(F) = \CF(F)\cap \RF(F)$ of row- and column-finite matrices in $\minff$.  The algebra $\RF$ is anti-isomorphic to $\CF$, via the transpose.  In terms of transformations, $\CF$ is isomorphic to the algebra of all linear transformations of a countably-infinite dimensional vector space over $F$ (with transformations written on the left of vectors).

The algebras $\RF$ and $\CF$ act on the left and right on $\minff$, respectively, making $\minff$ a left $\RF$-module and a right $\CF$-module.  (The associativity of products from $\RF\times\RF\times M$ and $M\times\CF\times\CF$ is checked in \cite[Proposition 1.1(2)(3)]{O2}.)  In particular, $M$ is a left and right $B(F)$-module.  Moreover, by \cite[Proposition 1.1(1)]{O2}, products from $\RF\times \minff\times\CF$ are associative (so that $\minff$ is an $(\RF,\CF)$-bimodule, and hence a $B(F)$-bimodule).  Let us refer to the associativity of such products as \emph{\RMC-associativity}, short for ``rowfinite$\cdot$middle$\cdot$columnfinite-associativity". 

In any of the above matrix spaces, as in the finite matrix algebras $M_n(F)$, we write $E_{ij}$ for the standard matrix units, where the $(i,j)$-entry of $E_{ij}$ is $1$ and all other entries are $0$.

Let $W := W_F$ denote the vector space of $\omega\times1$ column vectors over $F$, and let $V = V_F$ denote the subspace of $W$ consisting of column vectors with only finitely many nonzero entries.  The corresponding row spaces are obtained by transposing, so $W^t = \{ w^t \mid w \in W \}$ denotes the space of all $1\times\omega$ row vectors over $F$, and $V^t = \{ v^t \mid v\in V \}$ is the subspace of $W^t$ consisting of row vectors with only finitely many nonzero entries.  The space $W^t$ may be identified with the dual space $\Hom_F(V,F)$ of $V$.  Via matrix multiplication, $\RF$ acts on the left on $W$, while $\CF$ acts on the left on $V$.  Similarly, $\CF$ acts on the right on $W^t$, while $\RF$ acts on the right on $V^t$.

All of the above actions are associative, so that
$$
\begin{matrix}
\text{$V$ is a left $\CF$-module,} &\qquad\quad\text{$V^t$ is a right $\RF$-module,}\\
\text{$W$ is a left $\RF$-module,} &\qquad\quad\text{$W^t$ is a right $\CF$-module.}
\end{matrix}
$$
In particular, $V$, $V^t$, $W$, $W^t$ are all $B(F)$-modules.  Moreover, $V$ is a $B(F)$-sub\-mod\-ule of $W$ and $V^t$ is a $B(F)$-submodule of $W^t$.

The module-theoretic term ``annihilator" will be used in place of ``kernel", so that for a matrix $A \in \CF$,
\begin{align*}
\rann_V(A) &:= \{ v \in V \mid Av = 0 \}  \\
 \lann_{W^t}(A) &:= \{ f \in W^t \mid fA = 0 \}
\end{align*}
denote the right and left annihilators of $A$ in $V$ and $W^t$, respectively.  The notations $\rann_W(B)$ and $\lann_{V^t}(B)$, for $B \in \RF$, are parallel.

Given $v,u \in V$, we may view $v^t$ (resp., $u$) as the first row (resp., first column) of a matrix $x \in \RF$ (resp., $y \in \CF$) whose other rows (resp., columns) are zero.  The \RMC-associativity of $xAy$ for $A \in M$ reduces to $(v^tA)u = v^t(Au)$.

Eigenvalues for a matrix $A \in B(F)$ depend on the vector space in which eigenvectors may occur.  In particular, we must distinguish between eigenvalues for the action of $A$ on $V$ versus eigenvalues for the action of $A$ on $W$, i.e., between the existence of eigenvectors in $V$ versus eigenvectors in $W$.

Finally, a \emph{unit} in a ring (or algebra) with identity is any element with a multiplicative inverse in that ring.

                                %
                                %
\section{Treading warily with associativity in $M(F)$}  \label{sec:assoc}

Many of us mistakenly assume that multiplication of infinite matrices over a field $F$ must, as in the finite case, be associative (when all products involved make sense). Instead, we should be treading very warily when working with infinite matrices, because -- as a number of people over the years have observed -- we can have $(AB)C \neq A(BC)$ even when all the products $AB$, $(AB)C$, $BC$, $A(BC)$ are defined.

To reiterate the definedness of products in $M = M(F)$:
\begin{itemize}
\item[] The product $AB$ of matrices $A = (a_{ij})$ and $B = (b_{ij})$ in $M$ is defined if and only if 
\begin{itemize}
\item[$\bullet$] For all $i,k \in \Zpos$, there are at most finitely many $j \in \Zpos$ for which $a_{ij}b_{jk} \ne 0$.
\end{itemize}
\end{itemize}
One example of non-associativity is the following.  Other instances are given in Example \ref{l&r.inverse.not.match} and Remark \ref{no.weaken.BF.invert}.

\begin{example}  \label{kevin.nonassoc}
\cite[Example, p.2187]{O2}
Consider the following matrices in $M$:
$$
A :=
\renewcommand{\arraystretch}{1.2}
\begin{bmatrix}
 1&-1&1&-1&\cdots\\  0&0&0&0\\ 0&0&0&0\\ 0&0&0&0\\  \vdots &&&&\ddots \end{bmatrix}, \qquad
 B :=
\begin{bmatrix}
 1&1&0&0&\cdots\\  0&1&1&0\\ 0&0&1&1\\ 0&0&0&1\\  \vdots &&&&\ddots \end{bmatrix}, \qquad 
C := A^t\,.
$$
Then the products $AB$, $BC$, $(AB)C$, $A(BC)$ are all defined, and $(AB)C = E_{11}$ whereas $A(BC) = 0$.
\qed\end{example}

The set $M$ together with various operations is a type of \emph{partial nonassociative algebra over $F$}, with
\begin{itemize}
\item Fully defined addition and scalar multiplication operations, under which $M$ is a vector space over $F$.
\item A partially defined multiplication operation, which distributes over addition where defined, and with which scalar multiplication associates where products are defined.
\item An identity element.
\end{itemize}

Dating back to at least the 1950's, genuinely nonassociative $F$-algebras (as against partial algebras) of $\omega \times \omega$ matrices with well-defined matrix products have also been observed (see \cite{Ver} for instance). Building on examples such as \ref{kevin.nonassoc}, George Bergman recently gave a lovely new class of examples involving failure of even ``power associativity''.

\begin{example}  \label{ex.bergman2}  \cite{GB}
Note first that the set $M_4(M)$ of $4\times4$ matrices with entries from $M$ is a partial nonassociative $F$-algebra with respect to matrix operations together with the operations on $M$.  By a standard ``stretching argument'', $M_4(M) \cong M$ as partial algebras.

Choose any $A,B,C \in M$ such that $AB$, $BC$, $(AB)C$, $A(BC)$ are all defined but $(AB)C \ne A(BC)$, and set
$$
D :=
\renewcommand{\arraystretch}{1.2}
\begin{bmatrix}
0&A&0&0\\  0&0&B&0\\ 0&0&0&C\\ 0&0&0&0
\end{bmatrix} \in M_4(M).
 $$
Arbitrary products with all factors equal to $D$, however bracketed, are defined.  In particular,
\begin{align*}
D^2 &= \renewcommand{\arraystretch}{1.2}
\begin{bmatrix}
0&0&AB&0\\  0&0&0&BC\\ 0&0&0&0\\ 0&0&0&0
\end{bmatrix}, 
&\qquad D(D^2) &=
\renewcommand{\arraystretch}{1.2}
\begin{bmatrix}
0&0&0&A(BC)\\  0&0&0&0\\ 0&0&0&0\\ 0&0&0&0
\end{bmatrix},\\
(D^2)D &=
\renewcommand{\arraystretch}{1.2}
\begin{bmatrix}
0&0&0&(AB)C\\  0&0&0&0\\ 0&0&0&0\\ 0&0&0&0
\end{bmatrix},
&&\begin{matrix}
\text{products with $4$ or more}\\
\text{$D$ factors equal $0$.}
\end{matrix}
\end{align*}
Hence, the set $S := \{ I, D, D^2, D(D^2), (D^2)D, 0 \} \subseteq M_4(M)$ is closed under multiplication.  Since $(AB)C \ne A(BC)$, we have $(D^2)D \ne D(D^2)$, making $S$ nonassociative (hence a \emph{magma} rather than a monoid).

The $F$-span $R$ of $S$ is thus a nonassociative algebra (with fully defined multiplication) of dimension at most $5$.  Moreover, $R$ is isomorphic to a nonassociative subalgebra of $M$, due to the isomorphism $M_4(M) \cong M$.
\qed\end{example}

In spite of the bad behavior exhibited in examples such as those above, there are good instances of associativity, the most important of which is \RMC-associativity, as mentioned in \S\ref{NotaConv}.  We repeat the result below for emphasis.  Note, however, Example \ref{kevin.nonassoc} shows that ``column-finite $\times$ arbitrary $\times$ row-finite'' associativity fails. (The historical choice of matrix multiplication has a built-in bias of rows over columns -- rows go first in the inner products.)

\begin{proposition}  \label{RMC-assoc}
{\rm\cite[Proposition 1.1]{O2}}
If $A, B, C \in M$ with $A$ row-finite, $B$ arbitrary, and $C$ column-finite, then the products $AB$, $BC$, $(AB)C$, $A(BC)$ are all defined, and $(AB)C = A(BC)$.  In particular, $M$ is an $(\RF,\CF)$-bimodule.
\qed\end{proposition}

                                %
                                %
\section{$B(F)$ for general $F$}  \label{sec:BF}

We record some information on / conditions for invertibility, triangularity, and eigenvalues.  These results hold for arbitrary base fields $F$.  

\subsection{Invertibility}  \label{subsec:invert}
An immediate condition for invertibility in $B(F)$ is the following.

\begin{proposition}  \label{cor:units-go-well}
		Let $A\in M$. Suppose that there exist $X\in \RF$ and $Y\in \CF$ such that $XA=I=AY$. Then $X=Y\in B(F)$.  Hence, if $A$ is in $B(F)$ then it is invertible in $B(F)$.  
		
In particular, if $A \in B(F)$ is invertible in both $\RF$ and $\CF$, then $A$ is invertible in $B(F)$. 
	\end{proposition}

\begin{proof}
This is the standard argument for uniqueness of inverses, which is allowed by \RMC-associativity: $X = X(AY) = (XA)Y = Y$.
\end{proof}

The proposition fails, however, if $X$ and $Y$ are not assumed to be row-finite (resp., column-finite).
	
\begin{example}  \label{l&r.inverse.not.match}
There exist $A\in B(F)$ and $X,Y\in M$ such that $XA=I=AY$ but $X \ne Y$.  For example, start with
$$
A := 
\renewcommand{\arraystretch}{1.4}
\begin{bmatrix} 1&-1&0&0&\cdots\\  0&1&-1&0\\ 0&0&1&-1\\  \vdots &&&\ddots&\ddots \end{bmatrix} \in B(F),
$$
and take 
$$
X := 
\renewcommand{\arraystretch}{1.2}
\begin{bmatrix}
 1&1&1&1&\cdots\\  0&1&1&1\\ 0&0&1&1\\ 0&0&0&1\\  \vdots &&&&\ddots \end{bmatrix} \in \CF, \qquad
Y:=
\renewcommand{\arraystretch}{1.2}
\begin{bmatrix} 
0&0&0&0&\cdots\\  -1&0&0&0\\ -1&-1&0&0\\  -1&-1&-1&0\\ \vdots&&&&\ddots
\end{bmatrix} \in \RF.
$$
Note that $XA = AX = I$, so that $A$ is actually invertible in $\CF$.  It is not invertible in $\RF$, since $YA \ne I$.

We have here another example of non-associativity, since $(XA)Y \ne X(AY)$.
\qed\end{example}

In the situation of Proposition \ref{cor:units-go-well}, the assumption that $A$ has a two-sided inverse in $B(F)$ does not force $A$ to be in $B(F)$, as the following example of George Bergman shows.  The example also shows that a matrix with a two-sided inverse may have nonzero annihilators.

\begin{example}  \label{ex.bergman1}  \cite{GB}
There exist $A \in M$ and $X \in B(F)$ such that $XA = I = AX$ but $A \notin B(F)$.  Take
$$
A := \bigl( \min(i,j) \bigr) = 
\renewcommand{\arraystretch}{1.2}
\begin{bmatrix} 1&1&1&1&\cdots\\  1&2&2&2\\ 1&2&3&3\\ 1&2&3&4\\  \vdots &&&&\ddots \end{bmatrix}, \qquad
X := 
\renewcommand{\arraystretch}{1.2}
\begin{bmatrix} 2&-1&0&0&\cdots\\  -1&2&-1&0\\ 0&-1&2&-1\\ 0&0&-1&2\\  \vdots &&&&\ddots \end{bmatrix}.
$$
Note also that $X [1,2,3,\dots]^t = 0$.
\qed\end{example}

The following theorem is implicit in Proposition \ref{cor:units-go-well} but not in as revealing a form.

\begin{theorem}  \label{units.BF}
A matrix $A \in B(F)$ is a unit of $B(F)$ if and only if the column space of $A$ equals $V$ and the row space of $A$ equals $V^t$.
\end{theorem}

\begin{proof}
It is worth dwelling a moment on the proof of this important result, which does involve outside rings.  First, $A$ has a right inverse in $\CF$ iff the column space of $A$ equals $V$.  For the ``if" part, each standard basis vector $v_j$ in $V$, being a linear combination of columns of $A$, can be written $v_j = Ay_j$ for some $y_j \in V$, and we can take $y_j$ as the $j$-th column of a matrix $Y \in \CF$ to get a right inverse for $A$.  Conversely, if $AY = I$ for some $Y \in \CF$, then each $v_j$ is a linear combination of the columns of $A$ with coefficients from the $j$-th column of $Y$.

Similarly, $A$ has a left inverse $X$ in $\RF$ iff the row space of $A$ equals $V^t$.  The coup de grace follows from associativity as in Proposition \ref{cor:units-go-well}: $A$ has an inverse in $B(F)$ if and only if it has a left inverse in $\RF$ and a right inverse in $\CF$.
\end{proof}

Recall that the algebra $\CF = \End_F(V)$ is (\emph{von Neumann}) \emph{regular} (e.g., \cite[Proposition 4.27]{lam}), meaning that for any $A \in \CF$, there is some $X \in \CF$ such that $AXA = A$.  In a regular ring $R$ with identity, an element $a$ has a left (resp., right) inverse if and only if its right (resp., left) annihilator is zero.  To see ``if'', choose $x\in R$ with $axa = a$.  Then $a(xa-1) = 0$, so the condition $\rann_R(a) = 0$ implies $xa-1=0$, making $x$ a left inverse of $a$.  The reverse implication is clear.

\begin{lemma}  \label{invert.in.cfm}
A matrix $A \in \CF$ has a left inverse in $\CF$ if and only if $\rann_V(A) = 0$, while $A$ has a right inverse in $\CF$ if and only if $\lann_{W^t}(A) = 0$.
\end{lemma}

\begin{proof}
Since $\CF$ is regular, the matrix $A$ has a left inverse in $\CF$ if and only if $\rann_{\CF}(A) = 0$.  Note that $\rann_{\CF}(A)$ consists of those matrices in $\CF$ all of whose columns lie in $\rann_V(A)$.  Thus, $\rann_{\CF}(A) = 0$ if and only if $\rann_V(A) = 0$.

The proof for right inverses is analogous.
\end{proof}

\begin{corollary}  \label{BF.invertible}  
{\rm (\cite{Wan1}, \cite[Theorem A]{Wan2})}
A matrix $A \in B(F)$ is invertible in $B(F)$ if and only if $\rann_W(A)$ and $\lann_{W^t}(A)$ are both zero.
\end{corollary}

\begin{proof}
Suppose $A$ has an inverse $B \in B(F)$.  If $x \in W$ and $Ax = 0$, then using associativity for the $\RF$-module $W$, we obtain $x = BAx = 0$.  Similarly, if $y \in W^t$ and $yA = 0$, then using associativity for the $\CF$-module $W^t$ we obtain $y = yAB = 0$.  Thus the required annihilators are zero.

Conversely, assume that the stated annihilator conditions hold.  In particular, $\rann_V(A) = 0$ and $\lann_{W^t}(A) = 0$, so Lemma \ref{invert.in.cfm} shows that $A$ has an inverse $B \in \CF$.  By transposition, we also have $\rann_V(A^t) = 0$ and $\lann_{W^t}(A^t) = 0$, so Lemma \ref{invert.in.cfm} shows that $A^t$ has an inverse $C \in \CF$.  Hence, $C^t$ is an inverse for $A$ in $\RF$.  Using \RMC-associativity, we conclude that $C^t = C^tAB = B$, and therefore $B \in B(F)$.
\end{proof}

\begin{remark}  \label{no.weaken.BF.invert}
Neither of the annihilator conditions in Corollary \ref{BF.invertible} can be weakened with $W$ replaced by $V$.  For instance, take $A\in B(F)$ as in Example \ref{l&r.inverse.not.match}.  Since $A$ is invertible in $\CF$, we have $\rann_V(A) = 0$ and $\lann_{W^t}(A) = 0$ (Lemma \ref{invert.in.cfm}).  However, the inverse of $A$ in $\CF$, namely $X$, does not lie in the subalgebra $B(F)$, and so $A$ is not invertible in $B(F)$.

Symmetrically, $\rann_W(A^t) = 0$ and $\lann_{V^t}(A^t) = 0$, but $A^t$ is not invertible in $B(F)$.

The matrix $X \in B(F)$ of Example \ref{ex.bergman1} provides an example for which $\rann_V(X)\allowbreak = 0$ and $\lann_{V^t}(X) = 0$ but $X$ is not even right or left invertible in $B(F)$.  These annihilator conditions follow from $XA = I = AX$ (with $A\in M$ as in the example), due to the associativity for $M$ as a right or left $B(F)$-module.  On the other hand, $\rann_W(X)$ and $\lann_{W^t}(X)$ are both nonzero, for if $w := [1,2,3,\dots]^t \in W$, then $Xw = 0$ and $w^tX = 0$.  Thus, by Corollary \ref{BF.invertible}, $X$ is not invertible in $B(F)$.  In fact, Lemma \ref{invert.in.cfm} shows that $X$ cannot have a right inverse in $\CF$, and by symmetry it cannot have a left inverse in $\RF$. 

The last example also provides another instance of non-associativity, since $A(Xw) = 0 \ne (AX)w$ and $(w^tX)A = 0 \ne w^t(XA)$.
\qed\end{remark}
	
	We now obtain a result about invertibility of self-adjoint elements.
	
\begin{proposition}
		\label{prop:sa-invertibles}
Let $*$ be an involution on $F$, and let $\smallsharp$ be the $*$-transpose involution on $B(F)$.
Suppose that $A\in B(F)$ is self-adjoint with respect to $\smallsharp$, i.e., $A^{\#} = A$. Then $A$ is invertible in $B(F)$ if and only if it is invertible in $\CF$ if and only if it is right invertible in $\CF$.
	\end{proposition}
	
	\begin{proof}
Suppose that $A$ is right invertible in $\CF$, with right inverse $X$. Then $AX=I$ implies $X^{\#}A=I$, where $X^{\#}\in \RF$ is the $*$-transpose of $X$. By
Proposition \ref{cor:units-go-well} we get $X^{\#}=X\in B(F)$ and so $A$ is invertible in $B(F)$.  The remaining implications are immediate.
	\end{proof}
	
Note that it is right invertibility in $\CF$ that is important here, not left invertibility. In the upcoming Example \ref{E:no eigenvalues}, $X$ is a symmetric matrix in $B(F)$ which is left invertible in $\CF$ (because $\rann_V(X) = 0$) but is not right invertible in $\CF$.

\subsection{Upper triangular matrices}

\begin{lemma}  \label{ut.invertible}
Let $A \in \CF$ be upper triangular.  Then $A$ is invertible in $\CF$ if and only if $A$ has a left inverse in $\CF$, if and only if all the diagonal entries of $A$ are nonzero.  In this case, $A^{-1}$ is upper triangular.
\end{lemma}

\begin{proof}
Write $A = (a_{ij})$, and for $i>0$ let $r_i$ (resp., $c_i$) denote the $i$th row (resp., $i$th column) of $A$.

Suppose some $a_{nn} = 0$.  Set $Z_n := \{ z \in V \mid z_i=0\ \forall\ i>n \}$, and note (using $a_{nn} = 0$) that $r_i Z_n = 0$ for $i \ge n$.  Since $Z_n \cap \bigcap_{i=1}^{n-1} \rann_V (r_i)$ has codimension at most $n-1$ in $Z_n$, there is a nonzero vector $z \in Z_n$ such that $r_i z = 0$ for $i=1,\dots,n-1$.  Thus $r_i z = 0$ for all $i$, and so $Az = 0$.  Lemma \ref{invert.in.cfm} then shows that $A$ cannot have a left inverse in $\CF$.

Now assume all $a_{nn} \ne 0$.  If $z \in V$ is nonzero, there is an index $n$ such that $z_n \ne 0$ and $z_i = 0$ for all $i>n$, whence $r_n z = a_{nn} z_n \ne 0$ and $Az \ne 0$.  Thus $\rann_V(A) = 0$.  If $w \in W^t$ is nonzero, there is an index $n$ such that $w_n \ne 0$ and $w_j = 0$ for $j<n$, whence $wc_n = w_n a_{nn} \ne 0$, and so $wA \ne 0$.  Thus $\lann_{W^t}(A) = 0$. By Lemma \ref{invert.in.cfm}, $A$ is invertible in $\CF$.

This establishes the stated equivalences.

Now assume that $A$ is invertible in $\CF$.  If $A^{-1} = (b_{ij})$  is not upper triangular, $b_{ij} \ne 0$ for some $i>j$.  Without loss of generality, $b_{ik} = 0$ for $k<j$.  But then $ (A^{-1}A)_{ij} = b_{ij} a_{jj} \ne 0$, which is impossible.  Therefore $A^{-1}$ is upper triangular.
\end{proof}

\begin{remark}  \label{ut.notBFinvert}
If $A \in B(F)$ is upper triangular with nonzero diagonal entries, it need not be invertible in $B(F)$.  For example, take $A \in B(F)$ as in Example \ref{l&r.inverse.not.match}, so that $A$ is invertible in $\CF$ with inverse $X \notin B(F)$.  Since $B(F)$ is a subalgebra of $\CF$, the matrix $A$ cannot be invertible in $B(F)$.

Compare with Corollary \ref{BF.invertible}: $\lann_{W^t}(A)$ is zero, whereas $\rann_W(A)$ contains $[1,1,1,1,\dots]^t$.
\qed\end{remark}

\subsection{Eigenvalues}
We next exhibit some matrices in $B(F)$ which have no eigenvalues (or only one) for their action on $V$, even if $F$ is algebraically closed.  A simple example to keep in mind is the shift matrix
$$
A :=
\renewcommand{\arraystretch}{1.2}
\begin{bmatrix}  0&1&0&0&\cdots\\  0&0&1&0\\  0&0&0&1\\  \vdots&&&\ddots&\ddots
\end{bmatrix},
$$
since $0$ is the only eigenvalue for the action of $A$ on $V$,  while every $\lambda \in F$ is an eigenvalue for the action on $W$ (with eigenvector $[\lambda,\lambda^2,\lambda^3,\dots]^t$).

\begin{example}\label{E:no eigenvalues} Consider the symmetric, block tridiagonal matrix in $B(F)$
$$
Y_\lambda \ \ := \ \ 
\renewcommand{\arraystretch}{1.7}
\begin{bmatrix}
A & C & & & & \\
B & A & C &  & & \\
& B & A & C & & \\
&&   \ddots & \ddots & \ddots & &
\end{bmatrix}
$$
where
\[
A \ :=   \ \left[
\begin{array}{cc}
  \lambda & 1 \\
  1 &  \lambda
  \end{array} \right], \ \ \
B \ := \ \left[
\begin{array}{ccc}
  0 & 0 \\
  1 & 0
  \end{array} \right], \ \ \
C \ := \ \left[
\begin{array}{ccc}
  0 & 1 \\
  0 & 0
  \end{array} \right]
\]
and $\lambda \in F$.
It is easily checked that the columns of $Y_\lambda$ are linearly independent irrespective of the scalar $\lambda$, i.e., $\rann_V(Y_\lambda) = 0$, making $Y_\lambda$  a left invertible matrix of $\CF$ (Lemma \ref{invert.in.cfm}). 

So taking $X := Y_0$ gives a matrix with no eigenvalues (relative to its action on $V$) because the equation $(X - \mu I)v = 0$ for any scalar $\mu$, i.e., $Y_{-\mu}v = 0$, has no nonzero solutions $v \in V$. 

On the other hand, $\lann_{W^t}(X) \neq 0$ because $z = [0,1,0,-1,0,1,0,-1,\ldots ]$ left annihilates $X$, and so $X$ is not invertible in $\CF$.
\qed\end{example}

Our next example illustrates the contrast between actions on $V$ and on $W$, by giving a matrix $A \in B(F)$ for which the action on $V$ has no eigenvalues, while the action on $W$ has all scalars as eigenvalues.

\begin{example}  \label{tridiag.no.eigenval}
Take
$$
A := 
\renewcommand{\arraystretch}{1.5}
\begin{bmatrix} 1&1&0&0&0&\cdots\\  1&2&1&0&0\\  0&1&2&1&0\\  \vdots &&&\ddots \end{bmatrix} \in B(F).
$$
If $x \in W$ and $\lambda \in F$, then $Ax = \lambda x$ if and only if
\begin{equation}  \label{121eigenvec}
x_1+x_2 = \lambda x_1 \qquad\text{and}\qquad x_{k-1}+2x_k+x_{k+1} = \lambda x_k \quad \forall\; k>1.
\end{equation}
On one hand, \eqref{121eigenvec} requires $x_k = (\lambda-2)x_{k+1} - x_{k+2}$ for all $k>0$, which forces $x=0$ if $x \in V$.  Thus $A$ has no eigenvectors in $V$, so its action on $V$ has no eigenvalues.

On the other hand, if $0 \ne x_1 \in F$ and $x_2 := (\lambda-1)x_1$, while $x_3,x_4,\dots$ are chosen recursively with $x_k = (\lambda-2)x_{k-1} - x_{k-2}$ for $k \ge 3$, then $x \ne 0$ and \eqref{121eigenvec} is satisfied.  Therefore all $\lambda \in F$ are eigenvalues for the action of $A$ on $W$.

In particular, there exists some nonzero $x \in W$ with $Ax = 0$, and so $A$ is not invertible in $B(F)$ (Corollary \ref{BF.invertible}).  Moreover, $x^tA = (Ax)^t = 0$ with $x^t \in W^t$, so $A$ is not even invertible in $\CF$ (Lemma \ref{invert.in.cfm}).
\qed\end{example}

                                %
                                %
\section{Positive definite matrices}  \label{sec:posdef}

One of the most marketable classes of finite real or complex matrices consists of
the Hermitian matrices that are also positive definite (or semidefinite). They have many applications across a broad range of disciplines (see Chapter 7 of the Horn and Johnson text \cite{HJ}). In this section we examine the manners of some of their infinite cousins.

For Sections \ref{sec:posdef}--\ref{sec:squareroot}, we restrict our base field to $F = \RR$ or $F = \CC$, and we denote the standard involution on matrices by $*$, i.e., $*=$ transpose when $F=\RR$ and $*=$ conjugate transpose when $F=\CC$.

The standard definition of a positive definite real or complex matrix $A$ is a Hermitian (i.e., self-adjoint) $n \times n$ matrix that satisfies $x^*Ax > 0$ for all $0 \ne x \in F^n$ (as column vectors); equivalently, the inner product $\langle x, y \rangle := x^*Ay$ on $F^n$ is positive definite. Other characterizations are among the following equivalent conditions for $A \in M_n(F)$:
\begin{enumerate}
\item $A$ is positive definite.
\item $A$ is Hermitian and all principal minors of $A$ (i.e., the determinants of the upper left corner submatrices of $A$) are positive.
\item $A$ is Hermitian and all eigenvalues of $A$ are real and positive.
\item $A$ has a Cholesky decomposition, i.e., $A = U^tU$ for some upper triangular $U \in M_n(F)$ whose diagonal entries are all real and positive.
\item $A$ has a nonsingular Hermitian square root, i.e., $A = S^2$ for some invertible Hermitian $S \in M_n(F)$.
\end{enumerate}
Behind conditions (3) and (5) is the fact that a Hermitian matrix is unitarily diagonalizable. This is all standard fare in undergraduate linear algebra courses.

These conditions may also be considered for infinite matrices, even -- in the case of (1)--(4) -- for matrices in $\minfr$.  Of course, the matrices $U$ and $S$ in (4) and (5) must be allowed to be infinite.  

\begin{definition}
A matrix $A \in \minfr$ is \emph{positive definite} if $A$ is Hermitian and $x^*Ax > 0$ for all nonzero $x \in V$.  Note that $x^*Ax$ is well defined without parentheses due to \RMC-associativity.  The positivity condition does not make sense for $x \in W$, since then $x^*Ax$ is typically undefined, even if $A \in B(F)$ (e.g., if $A=I$).
\qed\end{definition}

It is known that conditions (1), (2), and (4) remain equivalent in the infinite setting (e.g., \cite[Lemma 3.1]{CWS}, \cite[Theorem 1]{Fin}).  We give different arguments for these equivalences in Lemma \ref{bigM.posdef.blocks} and Theorem \ref{posdef.vs.Chofac}.  However, conditions (3) and (5) are not equivalent to (1),(2),(4), even for Hermitian row- and column-finite matrices $A$, because:
\begin{itemize}
\item $A$ may have no eigenvalues at all for its action on $V$.
\item $A$ may have negative eigenvalues for its action on $W$.
\item $A$ may not be invertible.
\item $A$ may have no square roots.
\end{itemize}
We shall exhibit positive definite row- and column-finite matrices with these negative properties.  See Examples \ref{unit.posdef.no.eigenval}, \ref{2nd.unit.posdef.no.eigenval}, \ref{tridiag.posdef.no.eigenval}, \ref{tridiag.posdef.no.sqrt}, \ref{unit.posdef.no.sqrt}.

\begin{lemma}  \label{ctc.posdef}
{\rm(1)} For $C \in \CF$, the product $C^*C \in \minfr$ is positive definite if and only if $\rann_V(C) = 0$.

{\rm(2)} Let $A \in \minfr$ be positive definite.  If $A$ has an eigenvector $x \in V$, then the corresponding eigenvalue $\lambda$ is real and positive.
\end{lemma}

\begin{proof}
(1)  If $C^*C$ is positive definite, then $x^*(C^*C)x > 0$ for nonzero $x \in V$, whence $Cx \ne 0$.  On the other hand, if $Cx \ne 0$ for nonzero $x \in V$, then $x^*(C^*C)x = (Cx)^*(Cx) = \Vert Cx \Vert^2 > 0$.  (Note \RMC-associativity is used here.)

(2)  This holds since $x^*Ax = \lambda x^*x$ with $x^*Ax$ and $x^*x$ both real and positive.
\end{proof}

\begin{example}  \label{unit.posdef.no.eigenval}
There exist positive definite matrices $A \in B(F)$ such that $A$ is invertible in $B(F)$ and its action on $V$ has no eigenvalues, while all nonzero scalars in $F$ are eigenvalues for its action on $W$.

For example, let
$$
U := 
\renewcommand{\arraystretch}{1.5}
\begin{bmatrix} 1&1&1&0&0&0&\cdots \\  0&1&1&0&0&0 \\  0&0&1&1&1&0 \\  0&0&0&1&1&0 \\  \vdots &&&&\ddots \end{bmatrix}
$$ 
and take $A := U^*U = U^tU$.  Observe that $U$ is invertible in $B(F)$, with
$$
U^{-1} = 
\renewcommand{\arraystretch}{1.5}
\begin{bmatrix} 1&-1&0&0&0&0&\cdots \\  0&1&-1&1&0&0 \\  0&0&1&-1&0&0 \\  0&0&0&1&-1&1 \\  \vdots &&&&\ddots \end{bmatrix}.
$$
Consequently, $A$ is invertible in $B(F)$, and $A$ is positive definite by Lemma \ref{ctc.posdef}.  Now
$$
A = 
\renewcommand{\arraystretch}{1.5}
\begin{bmatrix} 1&1&1&0&0&0&0&0&\cdots \\
1&2&2&0&0&0&0&0 \\  1&2&3&1&1&0&0&0 \\
0&0&1&2&2&0&0&0 \\  0&0&1&2&3&1&1&0 \\
\vdots&&&&&\ddots  \end{bmatrix} .
$$
Row $2k$ of $A$ starts with $2k-2$ zeros, followed by $1,2,2$ and then zeros, while row $2k+1$ of $A$ (for $k>0$) starts with $2k-2$ zeros, followed by $1,2,3,1,1$ and then zeros.

Suppose $x\in V$ and $Ax = \lambda x$ for some $\lambda \in F$.  If $\lambda = 0$, then $x=0$ because $A$ is invertible.  Now assume that $\lambda \ne 0$, and fix $k>0$ such that $x_j = 0$ for $j>2k$.

Comparing entries $2k$ and $2k+1$ in $Ax = \lambda x$, we obtain $x_{2k-1}+2x_{2k} = \lambda x_{2k}$ and $x_{2k-1}+2x_{2k} = 0$, and hence $x_{2k-1} = x_{2k} = 0$.  Repeating this argument with $k$ replaced by $k-1,k-2,\dots,1$, we conclude that all $x_j=0$, whence $x=0$.

Therefore $A$ has no eigenvectors in $V$.
\medskip

Now let $\lambda \in F\setminus\{0\}$ be arbitrary.  We show there exists $x \in W$ with $Ax = \lambda x$.  Fix $x_1=1$ to start.  Comparing the first two entries in $Ax = \lambda x$ then yields the requirements
\begin{equation}  \label{x2x3}
x_2+x_3 = \lambda-1 \qquad\text{and}\qquad (2-\lambda)x_2+2x_3 = -1.
\end{equation}
Since $\det\begin{bmatrix} 1&1\\ 2-\lambda&2 \end{bmatrix} = \lambda \ne 0$, equations \eqref{x2x3} uniquely determine $x_2$ and $x_3$.  Now the third and fourth entries in $Ax = \lambda x$ require
$$
x_4+x_5 = -1-2x_2+(\lambda-3)x_3 \qquad\text{and}\qquad (2-\lambda)x_4+2x_5 = -x_3 \,,
$$
which uniquely determines $x_4$ and $x_5$.  Continuing recursively, there are unique $x_j$ for $j\ge2$ such that $Ax = \lambda x$.
\qed\end{example}

\begin{example}  \label{2nd.unit.posdef.no.eigenval}
There exist positive definite matrices $Z \in B(F)$ (necessarily non-invertible) such that $Z$ has no eigenvalues for its action on $V$, while all scalars in $F$ are eigenvalues for its action on $W$.

Take $X := Y_0$ as in Example \ref{E:no eigenvalues}, so that $X$ has no eigenvalues for its action on $V$, and set $Z := X^2 = X^*X \in B(F)$.  Since $\rann_V(X) = 0$, $Z$ is positive definite.

Suppose $\lambda$ is an eigenvalue of $Z$ for its action on $V$.  Then $\lambda > 0$ (Lemma \ref{ctc.posdef}), so $\lambda = \mu^2$ for some $\mu \in F$.  Since $X$ has no eigenvalues relative to $V$, the equation $(X-\mu I)(X+\mu I)v = (Z-\lambda I)v = 0$ with $v \in V$ implies, first, that $(X+\mu I)v = 0$, and then $v = 0$, a contradiction.

Therefore $Z$ has no eigenvalues relative to $V$.
\medskip

To compute eigenvectors in $W$, write $Z$ in block form as
$$
Z = 
\renewcommand{\arraystretch}{1.4}
\begin{bmatrix} S&I_2&0&0&\cdots \\  I_2&T&I_2&0 \\  0&I_2&T&I_2 \\
\vdots&&&\ddots \end{bmatrix}
$$
where $S := \begin{bmatrix} 2&0\\ 0&1 \end{bmatrix}$ and $T := \begin{bmatrix} 2&0\\ 0&2 \end{bmatrix}$.  Given $w \in W$, write $w = \begin{bmatrix} w_1\\ w_2\\ \vdots \end{bmatrix}$ where the $w_j$ are $2\times1$.  For $\lambda \in F$, we have $(Z - \lambda I)w = 0$ if and only if
\begin{equation}  \label{S.T.eqns}
w_2 = \lambda w_1 - Sw_1 \quad\text{and}\quad w_{k+1} = \lambda w_k - w_{k-1} - Tw_k \quad \forall\; k>1.
\end{equation}
Choosing $w_1\ne 0$ and then defining $w_2,w_3,\dots$ recursively via \eqref{S.T.eqns}, we obtain a nonzero $w \in W$ with $Zw = \lambda w$.

Therefore all $\lambda \in F$ are eigenvalues for the action of $Z$ on $W$.
\qed\end{example}

\begin{example}  \label{tridiag.posdef.no.eigenval}  
Take $A$ as in Example \ref{tridiag.no.eigenval}.  Then $A$ is Hermitian but not invertible in $B(F)$ (nor even in $\CF$), and $A$ has no eigenvalues relative to $V$, but all $\lambda \in F$ are eigenvalues for $A$ relative to $W$.  Now take
$$
U := 
\renewcommand{\arraystretch}{1.4}
\begin{bmatrix} 1&1&0&0&\cdots\\  0&1&1&0\\ 0&0&1&1\\  \vdots &&&\ddots \end{bmatrix} \in \BR,
$$
and observe that $U^*U = U^tU = A$.  By Remark \ref{ut.notBFinvert}, $U$ is invertible in $\CF$, and so $\rann_V(U) = 0$.  Then Lemma \ref{ctc.posdef} shows that $A$ is positive definite.  (One can also compute inductively that all principal minors of $A$ equal $1$.)
\qed\end{example}

As usual, a \emph{principal block} of an infinite matrix $A$ is an upper left $n\times n$ submatrix, for any $n\in \Zpos$, and a \emph{principal minor} of $A$ is the determinant of any principal block.  For convenience, we sketch the following elementary, known result.

\begin{lemma}  \label{bigM.posdef.blocks}
A matrix $A \in \minfr$ is positive definite if and only if all principal blocks of $A$ are positive definite, if and only if all principal minors of $A$ are positive.
\end{lemma}

\begin{proof}
Let $A_1,A_2,\dots$ denote the principal blocks of $A$.

Assume first that $A$ is positive definite, and let $n>0$.  If $x \in F^n$ is a nonzero column and $y \in V$ consists of $x$ followed by zeros, then $x^*A_nx = y^*Ay > 0$.  Thus $A_n$ is positive definite.

If all $A_n$ are positive definite, then reversing the above argument shows that $A$ is positive definite.

The remaining equivalence holds by applying the classical result to all $A_n$.
\end{proof}

                                %
                                %
\section{Cholesky factorization}  \label{sec:cholesky}

Here again $F=\RR$ or $\CC$, and $*$ denotes the standard involution.

\begin{definition} 
A \emph{Cholesky factorization} of $A\in \minfr$ is a factorization $A = U^*U$ where $U \in \CF$ is upper triangular with all diagonal entries real and positive. 
\qed\end{definition} 

In the finite case, every positive definite matrix $A \in M_n(F)$ has a unique Cholesky factorization (e.g., \cite[Theorem 7.63]{Axl}), a result first published in 1924.

It is known, at least in the real case, that positive definite matrices in $\minfr$ also have unique Cholesky factorizations -- e.g., \cite[Lemma 3.1]{CWS} or \cite[Theorem 1]{Fin}.  We present a proof (that works over both $\RR$ and $\CC$) following the idea of \cite[Section 8.3, Theorem 3]{Nic}, by induction on sizes of principal blocks. That is our preferred proof, since we will need the resulting block information in analyzing positive definite matrices in $B(F)$ (Theorem \ref{Ch.fac.inBR}).

\begin{theorem}  \label{posdef.vs.Chofac}
A matrix $A\in \minfr$ is positive definite if and only if $A$ has a Cholesky factorization. The Cholesky factorization of $A$ is unique.
\end{theorem}

\begin{proof}
$(\Longleftarrow)$:  If there is a Cholesky factorization $A = U^*U$, then $U$ is invertible in $\CF$ (Lemma \ref{ut.invertible}), whence $\rann_V(U) = 0$.  Then Lemma \ref{ctc.posdef} shows that $A$ is positive definite.

$(\Longrightarrow)$: Assume $A$ is positive definite, and let $A_1,A_2,\dots$ denote the principal blocks of $A$.  Since each $A_n$ is positive definite (Lemma \ref{bigM.posdef.blocks}), it has a Cholesky factorization $A_n = U_n^*U_n$ in $M_n(F)$.

\textbf{Claim 1:} Each $U_{n+1} = \begin{bmatrix} U_n &c_n\\ 0&d_n \end{bmatrix}$ for some $c_n \in F^n$ and $d_n \in \RR_{>0}$.

Write $A_{n+1} = \begin{bmatrix} A_n&c\\ c^*&d \end{bmatrix}$ for some $c \in F^n$ and $d \in \RR$.  Set $c_n := (U_n^*)^{-1}c$ and $d_0 := d - c_n^*c_n \in \RR$, and observe that
$$
\begin{bmatrix} U_n^*&0\\ c_n^*&1 \end{bmatrix}  \begin{bmatrix} U_n&c_n\\ 0&d_0 \end{bmatrix} = A_{n+1} \,.
$$
Now $0 < \det A_{n+1} = (\det U_n^*)(\det U_n) d_0 = |\det U_n|^2 d_0$, whence $d_0>0$.  Setting $d_n := \sqrt{d_0} > 0$, we obtain
\begin{equation}  \label{block.Ch}
\begin{bmatrix} U_n^*&0\\ c_n^*&d_n \end{bmatrix}  \begin{bmatrix} U_n&c_n\\ 0&d_n \end{bmatrix} = A_{n+1} \,.
\end{equation}
Since $\begin{bmatrix} U_n&c_n\\ 0&d_n \end{bmatrix}$ is upper triangular with positive diagonal entries, \eqref{block.Ch} is a Cholesky factorization of $A_{n+1}$.  By uniqueness, $\begin{bmatrix} U_n&c_n\\ 0&d_n \end{bmatrix} = U_{n+1}$, establishing Claim 1.

In view of Claim 1, there is a matrix $U \in \minfr$ such that for all $n>0$, the principal $n\times n$ block of $U$ equals $U_n$.  Since all these blocks are upper triangular with positive diagonal entries, the same holds for $U$.  In particular, $U \in \CF$.

\textbf{Claim 2:} $U^*U = A$.

For $n>0$, we have
$$
U^*U = \begin{bmatrix} U_n^*&0\\ (-)&(-) \end{bmatrix} \begin{bmatrix} U_n&(-)\\ 0&(-) \end{bmatrix} = \begin{bmatrix} U_n^*U_n&(-)\\ (-)&(-) \end{bmatrix} = \begin{bmatrix} A_n&(-)\\ (-)&(-) \end{bmatrix}
$$
for suitable blocks $(-)$,
and so the principal $n\times n$ blocks of $U^*U$ and $A$ coincide.  This establishes Claim 2.

Therefore $A = U^*U$ is a Cholesky factorization of $A$.  Suppose that $A = U_1^*U_1$ is another Cholesky factorization.  For $n>0$, let $U_{1,n}$ denote the principal $n\times n$ block of $U_1$.  Comparing principal $n\times n$ blocks in $A = U_1^*U_1$, we obtain $A_n = U_{1,n}^*U_{1,n}$.  This is a Cholesky factorization for $A_n$, hence $U_{1,n} = U_n$ by uniqueness.  Thus $U_1 = U$, and therefore the Cholesky factorization of $A$ is unique.
\end{proof}

\begin{theorem}  \label{Ch.fac.inBR}
Let $A \in B(F)$ be positive definite with Cholesky factorization $A = U^*U$.

{\rm(a)} $U \in B(F)$.

{\rm(b)} $A$ has a left inverse in $\CF$ and $U$ is invertible in $\CF$.

{\rm(c)} $A$ is invertible in $B(F)$ if and only if $A$ has a left inverse in $B(F)$, if and only if $A$ is invertible in $\CF$, if and only if $U$ is invertible in $B(F)$.
\end{theorem}

\begin{proof}
(a) Of course $U$, being upper triangular, is column-finite.

Let $A_n$, $U_n$, etc.~be as in the proof of Theorem \ref{posdef.vs.Chofac}.
Given $k>0$, we show that the first $k$ rows of $U$ are finite.  

Write $A = (a_{ij})$ and choose $m \ge k$ such that $a_{ij} = 0$ for all $i \le k$ and $j>m$.  Consider $n \ge m$.  Each of the first $k$ rows of $A$ is zero from column $m+1$ on, hence  zero from column $n+1$ on.  Thus, the nonzero parts of these rows lie in $A_n$.  Writing $A_{n+1} = \begin{bmatrix} A_n&c\\ c^*&d \end{bmatrix}$, the first $k$ entries of $c$ must be zero.  Since $(U_n^*)^{-1}$ is lower triangular, the first $k$ entries of $c_n := (U_n^*)^{-1}c$ are zero, whence the nonzero parts of the first $k$ rows of $U_{n+1} = \begin{bmatrix} U_n &c_n\\ 0&d_n \end{bmatrix}$ lie in $U_n$.

In particular, the first $k$ rows of $U_{m+1}$ are zero from column $m+1$ on.  It follows by induction that for $n\ge m+1$, the first $k$ rows of $U_{n+1}$ are zero from column $m+1$ on.  Therefore the first $k$ rows of $U$ are finite, as desired.

(b) Lemma \ref{ut.invertible} implies that $U$ is invertible in $\CF$.  Since $A$ is positive definite, $\rann_V(A) = 0$, and so by Lemma \ref{invert.in.cfm}, $A$ has a left inverse in $\CF$.

(c) If $U$ is invertible in $B(F)$, then so is $U^*$, and thus $A = U^*U$ as well.

If $A$ has a left inverse $B$ in $B(F)$, then $BU^*U = I$, whence $U^{-1} = BU^* \in B(F)$.  The remaining equivalences follow from Proposition \ref{prop:sa-invertibles}.
\end{proof}

\begin{example}
The factors in a Cholesky factorization $A = U^*U$ in $B(F)$ are not necessarily invertible in $B(F)$.  

For example, take
$$
A := 
\renewcommand{\arraystretch}{1.5}
\begin{bmatrix} 1&1&0&0&0&\cdots\\  1&2&1&0&0\\  0&1&2&1&0\\  \vdots &&&\ddots \end{bmatrix} \in B(F)
$$
as in Example \ref{tridiag.no.eigenval}.  As noted in Example \ref{tridiag.posdef.no.eigenval}, $A$ is positive definite and $A = U^*U$ where
$$
U := 
\renewcommand{\arraystretch}{1.4}
\begin{bmatrix} 1&1&0&0&\cdots\\  0&1&1&0\\ 0&0&1&1\\  \vdots &&&\ddots \end{bmatrix} \in B(F)
$$
and $U$ is invertible in $\CF$ but not in $B(F)$.  Moreover, $A = U^*U$ is a Cholesky factorization.

Note that by Theorem \ref{Ch.fac.inBR}, $A$ is not invertible in $B(F)$.  In fact, we have $[1,-1,1,-1,\dots] A = 0$, and so $A$ is not even invertible in $\CF$ (as the theorem says).
\qed\end{example}

Finite real or complex matrices with linearly independent columns have a \emph{QR-decom\-posi\-tion}, as uniquely a product of a matrix with orthonormal columns and an upper triangular matrix with positive diagonal entries.  In the square case, the first factor is an orthogonal matrix.  Some proofs of this are done as corollaries of the existence of Cholesky factorizations.  The same line can be carried out in $\CF$, carefully noting appropriate associativity in several steps.

\begin{theorem}  \label{QU.decomp}
Let $A \in \CF$ such that the columns of $A$ are linearly independent in $V$.  

{\rm(a)} $A = QU$ for some $Q,U \in \CF$ such that $Q^*Q = I$ and $U$ is upper triangular with real positive diagonal entries.  The matrices $Q$ and $U$ are uniquely determined by $A$.

{\rm(b)} If $A$ is invertible in $\CF$, then $Q$ is an invertible element of $B(F)$, and $Q^{-1} = Q^*$.

{\rm(c)} If $A$ is an invertible element of $B(F)$, then so are $Q$ and $U$.
\end{theorem}

\begin{proof}
(a) By hypothesis, $\rann_V(A) = 0$, and so $A^*A$ is positive definite (Lemma \ref{ctc.posdef}).  By Theorem \ref{posdef.vs.Chofac}, there is a Cholesky decomposition $A^*A = U^*U$, and $U$ is invertible in $\CF$ by Lemma \ref{ut.invertible}.  Set $Q := AU^{-1} \in \CF$, so that $A = QU$ and $Q^*A = (U^{-1})^* A^* A = (U^{-1})^* U^*U = U$ (using \RMC-associativity).  With further \RMC-associativity, we get $Q^*Q = (Q^*A)U^{-1} = I$.

Suppose also $A = Q_1U_1$ with $Q_1, U_1 \in \CF$ such that $Q_1^*Q_1 = I$ and $U_1$ is upper triangular with positive diagonal entries.  Using  \RMC-assoc\-ia\-tivity again, $U_1^*U_1 = U_1^*Q_1^*Q_1U_1 = A^*A$, so $U_1^*U_1$ is a Cholesky factorization of $A^*A$.  By uniqueness, $U_1 = U$, whence also $Q_1 = Q$.  Thus $Q$ and $U$ are unique.

(b) If $A$ is invertible in $\CF$, then so is $Q = AU^{-1}$, and by a final use of \RMC-associativity, $Q^* = Q^*QQ^{-1} = Q^{-1}$.  In particular, $Q = (Q^{-1})^*$ is row-finite, whence $Q \in B(F)$.  Then $Q^{-1} = Q^* \in B(F)$, and therefore $Q$ is invertible in $B(F)$.

(c) In view of (b), $U = Q^{-1}A \in B(F)$, and then also $U^{-1} = A^{-1}Q \in B(F)$.
\end{proof}

\begin{corollary}  \label{units.BR}
The units of $B(F)$ are the products $QU$ where $Q$ and $U$ are units of $B(F)$ with $Q^{-1} = Q^*$ and $U$ upper triangular with real positive diagonal entries.
\qed\end{corollary}

                                %
                                %
\section{Square roots}  \label{sec:squareroot}

Here again $F=\RR$ or $\CC$, and $*$ denotes the standard involution.

As recalled above, the positive definite matrices in $M_n(F)$ are exactly the squares of invertible Hermitian matrices in $M_n(F)$.  Even omitting invertibility, this fails for infinite matrices.  

To set up a first example, we start with the following.

\begin{lemma}  \label{almost.triang.no.sqrt}
Let $A = (a_{ij}) \in \RF$ with $a_{ij} = 0$ whenever $j \ge i+2$ and $a_{i,i+1} \ne 0$ for all $i$.  

{\rm(a)} Every $\lambda \in F$ is an eigenvalue for $A$ on $W$, and the corresponding eigenspace $W_\lambda$ is $1$-dimensional.

{\rm(b)}  If $A \in \RF(\RR)$, then $A$ has no square root in $\RF(\RR)$ or $\RF(\CC)$.
\end{lemma}

\begin{proof}
(a) Observe that a vector $w \in W$ satisfies $Aw = \lambda w$ if and only if 
$$
\lambda w_i = \sum_{j=1}^\infty a_{ij} w_j = \sum_{j=1}^{i+1} a_{ij} w_j \qquad \forall\; i,
$$
if and only if
\begin{equation}  \label{recur}
w_{i+1} = a_{i,i+1}^{-1} \left( \lambda w_i - \sum_{j=1}^i a_{ij} w_j \right) \qquad \forall\; i.
\end{equation}
Thus, choosing $w_1$ arbitrarily and defining $w_k$ for $k>1$ recursively via \eqref{recur}, we obtain $w \in W_\lambda$ with arbitrary $w_1$.  On the other hand, if $w,z \in W_\lambda$ with $w \ne 0$, it follows from \eqref{recur} that $w_1 \ne 0$ and $z = z_1 w_1^{-1} w$.  Thus $\dim W_\lambda = 1$, as claimed.

(b) Although the real case follows from the complex case, we give a separate, much easier proof for the case when $F = \RR$. 

Suppose there exists $S \in \RF(\RR)$ such that $S^2 = A$, and note that $S$ commutes with $A$.  Then $S W_\lambda \subseteq W_\lambda$ for all $\lambda \in \RR$, for if $x \in W_\lambda$, then $\lambda Sx = SAx = ASx$.

Specialize to $\lambda = -1$, and choose a nonzero $w \in W_{-1}$, so that $W_{-1} = \RR w$.  Since $S W_{-1} \subseteq W_{-1}$, we must have $Sw = \mu w$ for some $\mu \in \RR$.  But then $- w = A w = S^2 w = \mu^2 w$ and so $\mu^2 = -1$, which is impossible.

Therefore there is no $S \in \RF(\RR)$ such that $S^2 = A$.

Before addressing the complex case, we calculate some specific eigenvectors.  In the polynomial ring $\RR[z]$, define $f_j(z)$ for $j>0$ recursively by $f_1(z) := 1$ and
$$
f_j(z) := a_{j-1,j}^{-1} \left( z f_{j-1}(z) - \sum_{k=1}^{j-1} a_{j-1,k} f_k(z) \right) \qquad \forall\; j>1.
$$
Observe that each $f_j(z)$ has degree $j-1$ and leading coefficient $l_j$ where
$$
l_1 := 1 \qquad\text{and}\qquad l_j := (a_{12}a_{23} \cdots a_{j-1,j})^{-1} \quad \forall\ j>1.
$$
In view of \eqref{recur}, we see that
$$
w_\lambda := [f_1(\lambda), f_2(\lambda), f_3(\lambda), \dots]^t \ \ \text{lies in} \ \ W_\lambda \quad\forall\ \lambda \in \CC,
$$
so that $W_\lambda = \CC w_\lambda$.  

Now suppose $A$ has a square root in $\RF(\CC)$, say $S+iT$ where $S,T \in \RF(\RR)$.  Then:
\begin{itemize}
\item $S^2-T^2 = A$ and $ST+TS = 0$.
\item $S$ and $T$ commute with $A$, and $ST = -TS$.
\end{itemize}
Since $S$ and $T$ commute with $A$, they stabilize each $W_\lambda$, and so $Sw_\lambda = \sigma_\lambda w_\lambda$ and $Tw_\lambda = \tau_\lambda w_\lambda$ for some $\sigma_\lambda, \tau_\lambda \in \CC$.  Applying $S^2-T^2 = A$ and $ST+TS = 0$ to $w_\lambda$, we obtain $\sigma^2_\lambda - \tau^2_\lambda = \lambda$ and $\sigma_\lambda \tau_\lambda = 0$ for all $\lambda \in \CC$.  When $\lambda \in \RR$, we have $w_\lambda \in W_{\RR}$ and hence $\sigma_\lambda, \tau_\lambda \in \RR$.  Therefore
$$
\tau_\lambda = 0 \quad \forall\ \lambda \ge 0 \qquad\text{and}\qquad \sigma_\lambda = 0 \quad \forall\ \lambda \le 0.
$$

Assume that $S \ne 0$, say the $i$-th row of $S$ is nonzero.  Let $m$ be maximal such that $S_{im} \ne 0$.  Then
$$
0 = S_iw_\lambda = \sum_{j=1}^m S_{ij}(w_\lambda)_j = \sum_{j=1}^m S_{ij}f_j(\lambda) \qquad \forall\; \lambda \le 0,
$$
where $S_i$ is the $i$-th row of $S$.
Thus $\sum_{j=1}^m S_{ij}f_j(z)$ has infinitely many roots, so $\sum_{j=1}^m S_{ij}f_j(z) = 0$.  But 
$$
\sum_{j=1}^m S_{ij}f_j(z) = S_{im} l_m z^{m-1} + \text{\rm[lower terms]}$$
with $S_{im} l_m \ne 0$, which is impossible.  Therefore $S = 0$.

Similarly, $T = 0$.  But then $A = S^2-T^2 = 0$, another contradiction.

Therefore $A$ has no square root in $\RF(\CC)$.
\end{proof}

\begin{example}  \label{tridiag.posdef.no.sqrt}
There exist positive definite matrices in $\BR$ which have no square roots (positive definite or not, Hermitian or not) in $\BR$ or $\BC$.

For example, let
$$
A  := 
\renewcommand{\arraystretch}{1.4}
\begin{bmatrix} 1&1&0&0&\cdots\\ 1&2&1&0\\ 0&1&2&1\\ \vdots &&&\ddots \end{bmatrix} \in \BR
$$
as in Example \ref{tridiag.no.eigenval}, and recall from Example \ref{tridiag.posdef.no.eigenval} that $A$ is positive definite.  
Lemma \ref{almost.triang.no.sqrt} implies that $A$ has no square root in $\RF(\CC)$, nor in $\CF(\CC)$.
\qed\end{example}

The positive definite matrix $A$ in Example \ref{tridiag.posdef.no.sqrt} is not invertible in $\BR$, as noted in Example \ref{tridiag.no.eigenval}.  But invertible examples exist as well.

\begin{example}  \label{unit.posdef.no.sqrt}
There exist positive definite matrices in $\BR$ which are invertible in $\BR$ but have no square roots (of any kind) in $\RF(\CC)$ or $\CF(\CC)$.

For example, take
$$
A := 
\renewcommand{\arraystretch}{1.5}
\begin{bmatrix} 1&1&1&0&0&0&0&0&\cdots \\
1&2&2&0&0&0&0&0 \\  1&2&3&1&1&0&0&0 \\
0&0&1&2&2&0&0&0 \\  0&0&1&2&3&1&1&0 \\
\vdots&&&&&\ddots  \end{bmatrix} \in \BR
$$
as in Example \ref{unit.posdef.no.eigenval}; then $A$ is positive definite and invertible in $\BR$.  As outlined there, recursive conditions for a vector $x \in W$ to be an eigenvector for $A$ with eigenvalue $\lambda$ can be given as follows:
\begin{equation}  \label{recur2}
\begin{aligned}
\begin{bmatrix} 1&1\\ 2-\lambda&2 \end{bmatrix} \begin{bmatrix} x_2\\ x_3 \end{bmatrix} &= \begin{bmatrix} (\lambda-1)x_1\\ - x_1 \end{bmatrix}  \\
\begin{bmatrix} 1&1\\ 2-\lambda&2 \end{bmatrix} \begin{bmatrix} x_{2k+2}\\ x_{2k+3} \end{bmatrix} &= \begin{bmatrix} - x_{2k-1} - 2 x_{2k} + (\lambda-3)x_{2k+1}\\ - x_{2k+1} \end{bmatrix} \qquad \forall\; k \in \Zpos \,.
\end{aligned}
\end{equation}
As long as $\lambda \ne 0$, an eigenvector $x$ with arbitrary nonzero $x_1$ can be obtained from this recursion.  Moreover, if $Ax = \lambda x$ and $x_1=0$, then $x = 0$.  Thus:

Every nonzero $\lambda \in \RR$ (or in $\CC$) is an eigenvalue for $A$ on $W$, and the corresponding eigenspace $W_\lambda$ is $1$-dimensional.  As in the real case of Lemma \ref{almost.triang.no.sqrt}, it follows that there is no $S \in \RF(\RR)$ such that $S^2 = A$.

The complex case requires additional calculations.  First, when $\lambda \ne 0$, the equations \ref{recur2} can be solved by choosing $x_1$ arbitrarily and using $\begin{bmatrix} 1&1\\ 2-\lambda&2 \end{bmatrix}^{-1} = \lambda^{-1} \begin{bmatrix} 2&-1\\ \lambda-2&1\end{bmatrix}$.  Define polynomials $f_j(z) \in \RR[z]$ recursively by setting
\begin{align*}
f_1(z) &:= 1, \qquad f_2(z) := 2z-1, \qquad f_3(z) := z^2-3z+1\\
f_{2k+2}(z) &:= (2z-5)f_{2k+1}(z) - 4f_{2k}(z) - 2f_{2k-1}(z) \quad(k>0)\\
f_{2k+3}(z) &:= (z^2-5z+5)f_{2k+1}(z) + (4-2z)f_{2k}(z) + (2-z)f_{2k-1}(z) \quad(k>0).
\end{align*}
Note that $\deg f_j(z) = j-1$ for all $j$.

With the above choices, $w_\lambda := [1, \lambda^{-1}f_2(\lambda), \lambda^{-1}f_3(\lambda), \dots]^t$ is a nonzero vector in $W_\lambda$.  Now suppose that $S+iT$, for some $S,T \in \RF(\RR)$, is a square root of $A$.  As in Lemma \ref{almost.triang.no.sqrt}, $Sw_\lambda = 0$ for all $\lambda < 0$ and $Tw_\lambda = 0$ for all $\lambda > 0$.

Suppose the $i$-th row of $S$ is nonzero and $m$ is maximal with $S_{im} \ne 0$.  Then 
$$
0 = S_i w_\lambda = S_{i1} + \sum_{j=2}^m S_{ij} \lambda^{-1} f_j(\lambda) \qquad \forall\ \lambda < 0,
$$
and consequently $S_{i1} z + \sum_{j=2}^m S_{ij} f_j(z) = 0$.  If $m \ge 3$, we find that $S_{im}\cdot \lc(f_m) = 0$, which is impossible.  In case $m = 1$ or $2$, we have $S_{i1} z = 0$ or $S_{i1} z + S_{i2}(2z-1) = 0$, which are also impossible.  Thus $S = 0$.

Similarly, $T=0$, whence $A = S^2-T^2 = 0$, another contradiction.  Therefore $A$ has no square root in $\RF(\CC)$.  By symmetry, $A$ has no square root in $\CF(\CC)$ either.
\qed\end{example}

\begin{remark}
For finite positive definite matrices $A$, the standard proof of the existence of square roots is through $A$ being orthogonally or unitarily diagonalizable and having positive eigenvalues. A natural question is whether there is a proof through the Cholesky factorization (whose own proof need not involve eigenvalues). We are not aware of any such proof. In fact, to the contrary, Example \ref{unit.posdef.no.sqrt} suggests such a proof does not exist. For the Cholesky factorization also holds in $B(F)$ (by Theorem \ref{Ch.fac.inBR}) but square roots are lacking (Example \ref{unit.posdef.no.sqrt}), whereas any proof of square roots via the Cholesky factorization in the finite case would presumably carry over to $B(F)$. 
\qed\end{remark}

                                %
                                %
                                
 \section{Automorphisms and involutions of the algebra B(F)}
 
 Here we return to an arbitrary base field $F$, with specializations to $F = \RR$ or $\CC$ where noted, and where $*$ denotes transpose or conjugate transpose in those cases.  
 
 By an \emph{involution} of an algebra we mean only a ring anti-automorphism $\smallsharp$ of order $1$ or $2$, that is, an additive map satisfying the identities $(xy)^{\#} = y^{\#} x^{\#}$ and $(x^{\#})^{\#} = x$.  Such a map is an $F$-algebra involution if and only if it is $F$-linear.  On complex algebras, one considers \emph{conjugate linear} involutions $\smallsharp$, meaning that $(\alpha x)^{\#} = \overline{\alpha} x^{\#}$ for $\alpha \in \CC$ and $x$ in the algebra, where $\overline{\alpha}$ is the complex conjugate of $\alpha$.  In particular, the conjugate transpose involution $*$ on $B(\CC)$ or $M_n(\CC)$ is conjugate linear.

The richness of an algebraic structure is often reflected in the nature of its automorphisms, such as all being inner automorphisms. This is the case, for example, with the group automorphisms of the group of all permutations of a set (finite or infinite), and the algebra automorphisms of the algebra of all linear transformations of a vector space (of arbitrary dimension). It also holds for the algebra automorphisms of any finite-dimensional simple algebra $A$ over an arbitrary field when the center of $A$ consists of just scalar multiples of the identity (Skolem-Noether Theorem). Luckily for us, Nathan Jacobson in the 1940s established that many other algebras of transformations, including $B(F)$, join this list.  The importance to us of having all the automorphisms of $B(F)$ inner is that it allows us to use a lovely old trick of Jacobson \cite[\S IX.12]{Jac53} to characterize all the general algebra involutions of $B(F)$ in terms of the units of $B(F)$ and the transpose involution. (Concerning units of $B(F)$, see Theorem \ref{units.BF} and Corollary \ref{units.BR}.) For the record:

\begin{theorem}\label{T:B} {\rm[Jacobson]} All $F$-algebra automorphisms of $B(F)$ are inner.
\end{theorem}

\begin{proof} 
This follows from \cite[\S IX.11, Corollary 1]{Jac53} since $B(F) \cong \mathcal{L}(V^t \mid V)$ where $V^t$ is identified with the span in $\Hom_F(V,F)$ of a dual basis corresponding to the standard countable basis of $V$ (cf.\ ~\cite[Exercise 2, p.\ 263]{Jac53}).  It also follows from Jacobson's original result \cite[Theorem 32]{Jac45}, which implies that ring automorphisms of $B(F)$ are semilinearly inner.
\end{proof}

We use below the natural fact that transpose and inverse commute in $B(F)$ and $M_n(F)$.  It is worth noting that this fails for square matrices over noncommutative division rings \cite[Exercise 3, p.24]{Jac53}, and in fact for square matrices over any ring $R$ such that $R/J(R)$ is noncommutative \cite[Main Theorem 2.3]{GKKL}.

\begin{theorem}\label{T:AlgInvol} 
The $F$-algebra involutions on $R = B(F)$ or $M_n(F)$ are exactly the maps $\smallsharp$ of the form
 \[
     A \ \longmapsto \ A^{\#} \  = \  C^{-1}(A^t)C \ \ \forall\ A \in R
\]
for some $C \in R$ which is  a symmetric or skew symmetric unit of $R$ relative to transpose ($C^t = C$ or $C^t = - C$). In particular, $\smallsharp$ is the transpose operation followed by a conjugation.
\end{theorem}

\begin{proof}  (Following the argument of Jacobson in \cite[\S IX.12]{Jac53})

If $C$ is a symmetric or skew symmetric unit of $R$, it is easily checked that the rule $A^{\#} = C^{-1}(A^t)C$ defines an $F$-algebra involution on $R$.

Conversely, assume that we are given an $F$-algebra involution $\smallsharp$ on $R$.
Composing two involutions gives an automorphism, which from Theorem \ref{T:B} we know will be inner.  The automorphism  $A \mapsto (A^{\#})^t$ is inner, so for some unit $D \in R$,
\[
       D^{-1}AD \ =(A^{\#})^t  \ \ \text{and} \ \  A^{\#} = (D^t)(A^t)(D^{-1})^t, \ \ \forall\ A \in R.
 \]
Inasmuch as $(A^{\#})^{\#} = A$, a simple computation shows $(D^t)D^{-1}$ centralizes all of $R$, whence must be a scalar matrix $\alpha I$. Therefore, $D^t = \alpha D$, whence $D = (\alpha D)^t = \alpha D^t = \alpha^2 D$ and consequently $\alpha = 1, -1.$ Finally, taking $C = (D^{-1})^t = (D^t)^{-1}$ does the trick.
\end{proof}

\begin{theorem}  \label{T:ConjLinInvol}
The conjugate linear involutions on $R = B(\CC)$ or $M_n(\CC)$ are exactly the maps $\smallsharp$ of the form
 \[
A \ \longmapsto \     A^{\#} \  = \  C^{-1}(A^*)C \ \ \forall\ A \in R
\]
for some Hermitian unit $C \in R$. In particular, $\smallsharp$ equals $*$ followed by a conjugation.
\end{theorem}

\begin{proof}
If $C \in R$ is a Hermitian unit, it is easily checked that $C^{-1}(-)^*C$ is a conjugate linear involution on $R$.

Conversely, assume that we are given a conjugate linear involution $\smallsharp$ on $R$.
First apply the argument of Theorem \ref{T:AlgInvol} with transpose replaced by $*$, noting that if $\smallsharp$ is a conjugate linear involution on $R$, then $A \mapsto (A^{\#})^*$ is a $\CC$-algebra automorphism of $R$.  This produces a unit $C \in R$, with $C^* = \alpha C$ for some $\alpha \in \CC^\times$ with $|\alpha| = 1$, such that $A^{\#} = C^{-1}(A^*)C$ for $A \in R$.

We claim that there is a nonzero scalar $\gamma \in \CC$ such that $\overline{\gamma} = \alpha\gamma$.  If $\alpha = 1$, we can take $\gamma := 1$, so assume $\alpha \ne 1$.  Then $\alpha = a+ib$ for $a,b \in \RR$ with $a^2+b^2 = 1$ and $a \ne 1$.  In this case, $\gamma := b + i(a-1)$ satisfies $\gamma \ne 0$ and $\overline{\gamma} = \alpha\gamma$.

Now $\gamma^{-1}C$ is a unit of $R$ such that $(\gamma^{-1}C)^* = \overline{\gamma}^{-1} C^* = (\alpha\gamma)^{-1} \alpha C = \gamma^{-1}C$.  Since $(\gamma^{-1}C)^{-1}A^*(\gamma^{-1}C) = C^{-1}A^*C = A^{\#}$ for $A \in R$, we are done on replacing $C$ with $\gamma^{-1}C$.
\end{proof} 

Recall: an involution $\smallsharp$ on a ring $R$ is \emph{positive definite}, or \emph{strictly proper}, if $\sum_{i=1}^n a_i^{\#}a_i = 0$ implies all $a_i = 0$, for any $a_i \in R$; it is \emph{proper} if $a^{\#}a = 0$ implies $a = 0$.  The identity involution on any field is proper, but on $\mathbb{C}$ it is not positive definite.  Of course transpose is positive definite on $M_n(\RR)$ and on $\BR$, while conjugate transpose is positive definite on $M_n(\CC)$ and on $\BC$.

But are the other involutions of $B(F)$ obtained as in Theorem \ref{T:AlgInvol} (resp., Theorem \ref{T:ConjLinInvol}) really that much different than transposing (resp., than $*$)? Parallels with the algebra $M_n(F)$ suggest they can be very different if we are willing to forsake positive definiteness.

\begin{example} Consider the real algebra $R = M_n(\mathbb{R})$ with $n \ge 2$ and define $\smallsharp$ by transposing a matrix across its NE--SW diagonal instead of the usual NW--SE diagonal; e.g.,
\[
\renewcommand{\arraystretch}{1.2}
      \left[\begin{array}{cccc}
      a & b & c & d\\
      e & f & g & h\\
      i & j & k & l\\
      m & n & o & p
      \end{array}\right]^{\#} \ \ \ \ = \ \ \ \
      \left[\begin{array}{cccc}
      p & l & h & d\\
      o & k & g & c\\
      n & j & f & b\\
      m & i & e & a
      \end{array}\right].
 \]
 This fits the bill for an involution because it takes the form $A^{\#} = C^{-1}A^tC$ for (in the $4\times4$ case)
 \[
\renewcommand{\arraystretch}{1.2}
 C \ \ := \ \ \left[\begin{array}{cccc}
      0 & 0 & 0 & 1\\
      0 & 0 & 1 & 0\\
      0 & 1 & 0 & 0\\
      1 & 0 & 0 & 0
      \end{array}\right].
 \]
 With respect to the standard involution all symmetric elements are well-behaved (e.g.\ diagonalizable) and therefore there are no nonzero symmetric nilpotent members of $R$. With this new involution, there are plenty of nonzero self-adjoint nilpotents! For example, the matrix unit $A = E_{1n}$ is symmetric but nilpotent.  Also the new involution is no longer positive definite because $AA^{\#} = 0$ for $A = E_{11}$.
\qed\end{example}

\begin{definition}  \label{D:EquivInvol}
We call two involutions $\smallsharp$ and $\smallsharpp$ of an $F$-algebra $R$  \emph{equivalent} if there is an $F$-algebra automorphism $\theta$ of $R$ such that
\smallskip

\centerline{
     $\theta(a^{\#}) \ = \ (\theta(a))^{\#'} \ \ \forall\ a \in R. \qquad\qed$}
\end{definition}

When involutions $\smallsharp$ and $\smallsharpp$ on an $F$-algebra $R$ are equivalent, elements $a\in R$ and $\theta(a)$ have identical properties relative to the algebra and involution operations,  such as being self-adjoint or projections.  In particular, if $\smallsharpp$ is proper then so is $\smallsharp$.  Namely, if $a \in R$ with $a^{\#}a = 0$, then $0 = \theta(a^{\#}a) = (\theta(a))^{\#'} \theta(a)$, so $\theta(a) = 0$ and $a = 0$.  Similarly, if $\smallsharpp$ is positive definite then so is $\smallsharp$.

A case of particular interest to us is when the two involutions $\smallsharp$, $\smallsharpp$ are defined as in Theorems \ref{T:AlgInvol}, \ref{T:ConjLinInvol}.

\begin{proposition}  \label{P:EquivInvol}
  Let $R$ be $B(F)$ or $M_n(F)$. Let $C,D \in R$ be units, each either symmetric or skew symmetric relative to transpose, and let $\smallsharp$, $\smallsharpp$ be the $F$-algebra involutions defined as in Theorem \ref{T:AlgInvol} using $C$, $D$ respectively.

Then $\smallsharp$ and $\smallsharpp$ are equivalent if and only if there are a unit $U$ of $R$ and nonzero scalar $\alpha \in F$ such that  $C \ = \ \alpha U^t DU$.
\end{proposition}

\begin{proof}
 Since $F$-algebra automorphisms are inner here,
  \begin{align*}
\smallsharp\ &\text{and}\ \smallsharpp\ \text{are equivalent} \\
        &\Longleftrightarrow  D^{-1}(UAU^{-1})^tD \ = \ U(C^{-1}A^tC)U^{-1}  \ \ \mbox{for some unit $U \in R$ and all $A \in R$} \notag \\
        &\Longleftrightarrow  D^{-1}(U^t)^{-1}A^tU^tD \ = \ U(C^{-1}A^tC)U^{-1}  \ \ \mbox{for all $A \in R$} \notag \\
        &\Longleftrightarrow  CU^{-1}D^{-1}(U^t)^{-1} \ \ \mbox{centralizes $A^t$ for all $A \in R$ }  \notag \\
        &\Longleftrightarrow  CU^{-1}D^{-1}(U^t)^{-1} \ \ \mbox{is a nonzero scalar $\alpha$ times $I$}  \notag \\
        &\Longleftrightarrow  C \ = \ \alpha U^t DU \ \ \mbox{for some scalar $\alpha \in F^*$,}
  \end{align*}
as claimed.
\end{proof}

\begin{proposition}  \label{P:EquivConjLinInvol}
Let $R$ be $\BC$ or $M_n(\CC)$, and let $C,D \in R$ be Hermitian units.  Let $\smallsharp$, $\smallsharpp$ be the conjugate linear involutions on $R$ defined as in Theorem \ref{T:ConjLinInvol} using $C$, $D$ respectively.

Then $\smallsharp$ and $\smallsharpp$   are equivalent if and only if there are a unit $U$ of $R$ and nonzero scalar $\alpha \in \CC$ such that  $C \ = \ \alpha U^* DU$.
\end{proposition}

\begin{proof}
As Proposition \ref{P:EquivInvol}.
\end{proof}

 \begin{theorem}\label{T:EquivInvol}
  Let $R$ be $\BR$ or $M_n(\RR)$. Let $C \in R$ be a symmetric or skew symmetric unit relative to transpose, and let $\smallsharp$ be the $\RR$-algebra involution on $R$ defined as in Theorem \ref{T:AlgInvol} using $C$.
\begin{enumerate}
 \item  The involution $\smallsharp$ is equivalent to the transpose involution exactly when $C$ is a positive definite real matrix or its negative.
 \item If $C$ is a symmetric unit of $R$ but neither $C$ nor $-C$ is positive definite, then the involution $\smallsharp$ is not a proper involution. 
 \item To within equivalence, transposition is the only proper $\RR$-algebra involution on $R$.
 \item Every proper $\RR$-algebra involution on $R$ is positive definite.
  \end{enumerate}
  \end{theorem}

  \begin{proof}
 Specialising Proposition \ref{P:EquivInvol} to the case $F = \RR$ and $D = I$, we see that the unit $C$ will lead to an involution $\smallsharp$ equivalent to the transpose involution  exactly when $C = \alpha U^t U$ for some $\alpha \in \RR^\times$ and some invertible $U \in R$. Absorbing $\sqrt{\alpha}$ or $\sqrt{-\alpha}$ into $U$, we see that $\smallsharp$ is equivalent to transpose if and only if $C = \pm U^tU$ for some invertible $U \in R$. But one characterization of positive definitive matrices in $M_n(\RR)$ is that they can be written as $T^tT$ for some invertible $T \in M_n(\RR)$.  The same characterization holds for invertible positive definite elements of $\BR$ by Lemma \ref{ctc.posdef} and Theorems \ref{posdef.vs.Chofac}, \ref{Ch.fac.inBR}.  This gives (1).

 For (2), first consider the case of $C \in M_n(\RR)$. Since $\pm C$ are invertible but not positive definite, $C$ has eigenvalues $\lambda > 0$ and $\mu < 0$. To show $\smallsharp$ is not proper it is enough to produce a nonzero matrix $A$ with $A^tCA = 0$ for then $A^{\#}A = C^{-1}(A^tCA) = 0$. But since $C$ is real symmetric it is orthogonally diagonalizable, whence in the preceding condition we can assume $C= \diag(\lambda, \mu,\ldots, \omega )$. Take $A$ to be the matrix whose top left $2 \times 2$ block is
 \[
\left[
\begin{array}{cc}
1 & 1 \\
  &                     \\
\sqrt{\frac{-\lambda}{\mu}} & \sqrt{\frac{-\lambda}{\mu}} \\
\end{array}\right]
\]
 and all other entries zero. One checks $A^tCA = 0$.

 Next suppose $C \in \BR$. Since $\pm C$ are not positive definite, there exist $0 \ne x,y \in V$ with $x^tCx \le 0$ and $y^tCy \ge 0$. Choose $m \ge 2$ such that entries of $x$ and $y$ beyond the $m$th are 0, and let $z := [x_1,x_2,\ldots, x_m]^t$, $w := [y_1,\dots,y_m]^t$. Let $C_m$ be the upper left $m \times m$ corner of $C$. We have $C_m$ is symmetric and $z^tC_mz = x^tCx \le 0$, while $w^tC_mw = y^tCy \ge 0$.  Thus, $C_m$ and $-C_m$ are not positive definite. The corner $C_m$ may or not be invertible. If it is, our earlier argument provides a nonzero $A_m \in M_m(\RR)$ with $A^t_mC_mA_m = 0$. On the other hand, if $C_m$ is not invertible it has a nonzero right annihilator $r \in \RR^m$ . Let $A_m \in M_m(\RR)$ have first column $r$ and other columns zero. Then $C_mA_m = 0$ and $A_m^tC_mA_m = 0$. In either case, let $A \in \BR$ have $A_m$ as its upper left corner and and all other entries 0. We have $A^tCA = 0$ which implies $A^{\#}A = C^{-1}A^tCA = 0$, showing that the involution $\smallsharp$ produced from $C$ is not proper.
 
 Regarding (3), note that as $C$ varies, $\smallsharp$ covers all $\RR$-algebra involutions on $R$ (Theorem \ref{T:AlgInvol}).  If $\smallsharp$ is equivalent to transpose, then $\smallsharp$ is proper, in fact, positive definite, as noted following Definition \ref{D:EquivInvol}.
 
 Conversely, assume $\smallsharp$ is proper.
  If $C$ is skew symmetric, its diagonal entries are $0$.  In particular, $E_{11}CE_{11} = 0$ and $E^{\#}_{11} E_{11} = C^{-1}E_{11}CE_{11} = 0$, contradicting the properness of $\smallsharp$.  Hence, $C$ must be symmetric.  Now by (2), either $C$ or $-C$ is positive definite, and therefore (1) says that $\smallsharp$ is equivalent to transpose.
  
  (4) is immediate from (3) and its proof.
 \end{proof}
 
 \begin{corollary}  \label{proper.inv.BR}
Let $R$ be $\BR$ or $M_n(\RR)$.  If $\smallsharp$ is a proper $\RR$-algebra involution on $R$, then there is an invertible positive definite matrix $C \in R$ such that $A^{\#} = C^{-1}A^tC$ for all $A \in R$.
\end{corollary}

\begin{proof}
By Theorem \ref{T:AlgInvol}, there is a symmetric or skew-symmetric unit $C \in R$ such that $A^{\#} = C^{-1}A^tC$ for $A \in R$, and the argument of Theorem \ref{T:EquivInvol}(3) shows that $C$ is symmetric.  Hence, by Theorem \ref{T:EquivInvol}(2), either $C$ or $-C$ must be positive definite.  We may replace $C$ by $-C$ if necessary, since $C^{-1}A^tC = (-C)^{-1} A^t (-C)$ for $A \in R$.
\end{proof}

By way of an example, if we choose the symmetric unit
\[
\renewcommand{\arraystretch}{1.2}
C \ := \ \left[
\begin{array}{ccc}
0 & 0 & 1 \\
0 & 1 & 0 \\
1 & 0 & 0
\end{array}\right] \ \in \ M_3(\RR),
\]
which as we saw earlier leads to the involution of transposing in the NE--SW diagonal, by (1) of Theorem \ref{T:EquivInvol} this involution cannot be equivalent to the usual transpose because $C$ is not a positive definite matrix (its first two principal minors are not positive). On the other hand,
\[
\renewcommand{\arraystretch}{1.2}
C' \ := \ \left[
\begin{array}{cccc}
1 & 1 & 0 & 0 \\
1 & 2 & 0 & 0 \\
0 & 0 & 1 & 0 \\
0 & 0 & 0 & 1
\end{array}\right] \ \in \ M_4(\RR)
\]
is positive definite (all its principal minors are positive) so its corresponding involution is equivalent to the usual transpose.

\begin{theorem}  \label{T:Equiv*}
Let $R$ be $\BC$ or $M_n(\CC)$, and let $C \in R$ be a Hermitian unit.  Let $\smallsharp$ be the conjugate linear involution on $R$ defined as in Theorem \ref{T:ConjLinInvol} using $C$.
\begin{enumerate}
\item The involution $\smallsharp$ is equivalent to the conjugate transpose involution $*$ if and only if $C$ or $-C$ is positive definite.
\item If neither $C$ nor $-C$ is positive definite, then the involution $\smallsharp$ is not a proper involution. 
 \item To within equivalence, conjugate transpose is the only proper conjugate linear involution on $R$.
 \item Every proper conjugate linear involution on $R$ is positive definite.\end{enumerate}
\end{theorem}

\begin{proof}
(1) Taking $D = I$ in Proposition \ref{P:EquivConjLinInvol}, we see that $\smallsharp$ and $*$ are equivalent if and only if there are a unit $U \in R$ and a nonzero scalar $\beta \in \CC$ such that $C = \beta U^* U$.  Such a $\beta$ must be real, since $C$ and $U^*U$ are Hermitian.  After absorbing $\sqrt\beta$ or $\sqrt{-\beta}$ into $U$, we conclude that $\smallsharp$ and $*$ are equivalent if and only if $C = \pm U^*U$ for some invertible $U \in R$.  By Lemma \ref{ctc.posdef} and Theorems \ref{posdef.vs.Chofac}, \ref{Ch.fac.inBR}, the positive definite units in $\BC$ are the products $T^*T$ for units $T$ in $\BC$.  The corresponding statement for $M_n(\CC)$ is a classical fact.  This establishes (1).

(2) First assume that $C \in M_n(\CC)$.  Since $\pm C$ are Hermitian and invertible, all their eigenvalues are nonzero and real.  But neither is positive definite, so $C$ has eigenvalues $\lambda>0$ and $\mu<0$.  To show $\smallsharp$ is not proper, it is enough to produce a nonzero matrix $A$ with $A^*CA = 0$, for then $A^{\smallsharp}A = C^{-1}(A^*CA) = 0$.  As $C$ is Hermitian, it is unitarily diagonalizable, hence we may assume that $C = \diag(\lambda,\mu,\dots,\omega)$.  As in the proof of Theorem \ref{T:EquivInvol}(2), we take $A$ to be the matrix whose top left $2\times2$ block is
 \[
\left[
\begin{array}{cc}
1 & 1 \\
  &                     \\
\sqrt{\frac{-\lambda}{\mu}} & \sqrt{\frac{-\lambda}{\mu}} \\
\end{array}\right]
\]
 with all other entries zero, and we check that $A^*CA = A^tCA = 0$.
 
Now assume that $C \in \BC$.  Since $\pm C$ are not positive definite, there are nonzero $x,y \in V$ such that $x^*Cx \le 0$ and $y^*Cy \ge 0$.  Choose $m\ge2$ so that $x_i = y_i = 0$ for all $i>m$, and let $z$ and $w$ be the length-$m$ truncations of $x$ and $y$.  Let $C_m$ be the upper left $m\times m$ corner of $C$.  Then $C_m$ is Hermitian, $z^*Cz \le 0$, and $w^*Cw \ge 0$, so $C_m$ and $-C_m$ are not positive definite.

If $C_m$ is invertible, our previous argument provides a nonzero $A_m \in M_m(\CC)$ with $A_m^*C_mA_m = 0$.  If $C_m$ is not invertible, there is a nonzero column vector $r \in \CC^m$ with $C_mr = 0$.  In this case, take $A_m \in M_m(\CC)$ to have first column $r$ and all other columns zero, whence $C_mA_m = 0$ and $A_m^*C_mA_m = 0$.  In either case, let $A \in \BC$ have upper left $m\times m$ corner $A_m$ and all other entries zero, and observe that $A^*CA = 0$.  Thus $A^{\smallsharp}A = 0$.

(3) As $C$ varies, $\smallsharp$ covers all conjugate linear involutions on $R$.  If $\smallsharp$ is equivalent to $*$, then $\smallsharp$ is proper and thus positive definite, as noted following Definition \ref{D:EquivInvol}.

Conversely, assume $\smallsharp$ is proper.  By (2), either $C$ or $-C$ is positive definite, and therefore (1) says that $\smallsharp$ is equivalent to $*$.

(4) is immediate from (3) and its proof.
\end{proof}

 \begin{corollary}  \label{proper.inv.BC}
Let $R$ be $\BC$ or $M_n(\CC)$.  If $\smallsharp$ is a proper conjugate linear involution on $R$, then there is an invertible positive definite matrix $C \in R$ such that $A^{\#} = C^{-1}A^*C$ for all $A \in R$.
\end{corollary}

\begin{proof}
By Theorem \ref{T:ConjLinInvol}, there is a Hermitian unit $C \in R$ such that $A^{\smallsharp} = C^{-1}A^*C$ for $A \in R$.  Hence, Theorem \ref{T:Equiv*}(2) implies $C$ or $-C$ is positive definite.  Since $C$ may be replaced by $-C$ if necessary, we conclude with $C$ being positive definite.
\end{proof}

\begin{ack*}
We thank George Bergman for numerous comments and suggestions, as well as for Examples \ref{ex.bergman2} and \ref{ex.bergman1} and permission to present them here.  We thank Daniel Bossaller for alerting us to the fact that automorphisms of $B(F)$ are inner.  Also, special thanks to the referee for an extremely thorough reading of the manuscript, catching some errors and suggesting corrections.
\end{ack*}



\begin{thebibliography}{99}

\bibitem{Axl} S. Axler, \emph{Linear Algebra Done Right}, 4th. Ed., Cham (2024) Springer.

\bibitem{GB} G. Bergman, Private correspondence, 2026.

\bibitem{CWS} C.K. Chui, J.D. Ward, and P.W. Smith, \emph{Cholesky factorization of positive definite bi-infinite matrices}, Numerical Func. Anal. and Optimization \textbf{5} (1982), 1--20.

\bibitem{Fin} B. Finta, \emph{The $LL^T$ factorization for infinite matrices}, AIP Conf. Proc. \textbf{1148} (2009), 778--780.

\bibitem{GMM} K.R. Goodearl, P. Menal, and J. Moncasi, \emph{Free and residually artinian regular rings}, J. Algebra \textbf{156} (1993), 407--432.

\bibitem{GKKL} R.N. Gupta, A. Khurana, D. Khurana, T.-Y. Lam, \emph{Rings over which the transpose of every invertible matrix is invertible}, J. Algebra \textbf{322} (2009), 1627--1636.

\bibitem{HJ} R.A. Horn and C.R. Johnson, \emph{Matrix Analysis}, 2nd. Ed., Cambridge (2013) Cambridge Univ. Press.

\bibitem{Jac45} N. Jacobson, \emph{The radical and semi-simplicity for arbitrary rings}, Amer. J. Math. \textbf{67} (1945), 300--320.

\bibitem{Jac53} N. Jacobson, \emph{Lectures in Abstract Algebra II -- Linear Algebra}, Princeton (1953) Van Nostrand.

\bibitem{lam} T.Y. Lam, \emph{A First Course in Noncommutative Rings}, 2nd. ed., Grad. Texts in Math. \textbf{131}, New York (2001) Springer-Verlag.

\bibitem{Nic} W.K. Nicholson, \emph{Linear Algebra with Applications}, 5th. ed., Toronto (2006) McGraw-Hill Ryerson.

\bibitem{OM2003} K.C. O'Meara, \emph{The exchange property for row and column-finite matrices}, J. Algebra \textbf{268} (2003), 744--749.

\bibitem{O2} K.C. O'Meara, \emph{A new setting for constructing von Neumann regular rings}, Comm. Alg. \textbf{45} (2017), 2186--2194.

\bibitem{Ver} P. Vermes, \emph{Non-associative rings of infinite matrices}, Nederl. Akad. Wetensch. Proc. Ser. A \textbf{55}, Indag. Math. \textbf{14} (1952), 245--252.

\bibitem{Wan1}  S.Q. Wang, \emph{On the inverses of infinite matrices over a field}, J. Beijing Normal Univ. (Chinese) \textbf{29} (1993), 327--330.

\bibitem{Wan2} S.Q. Wang, \emph{Diagonalization of row-column-finite infinite matrices}, Sci. China (Ser. A) \textbf{40} (1997), 1279--1286.

\end{thebibliography}
\end{document}